\documentclass[12pt]{amsart}

\usepackage[utf8x]{inputenc}
\usepackage[english]{babel}
\usepackage{amsmath, amsfonts, amssymb}
\usepackage{color}
\usepackage{tikz}
\usetikzlibrary{positioning,shapes,shadows,arrows}
\usepackage{enumitem}
\usepackage{hhline}
\usepackage{array}
\usepackage[colorlinks=true,linkcolor=blue,citecolor=blue,urlcolor=blue]{hyperref}
\usepackage{cite}
\usepackage{listings}
\usepackage{shuffle}
\usepackage{todonotes}
\usepackage{frcursive}
\usepackage{mathtools}
\usepackage{pdflscape}
\usepackage{amsaddr}
\usepackage{todonotes}
\usepackage{xspace}

\numberwithin{equation}{section}

\title{A description of the Zeta map on Dyck paths area sequences}

\author{Viviane Pons}

\address{CNRS, IRL CRM Montréal -- Université Paris-Saclay, CNRS, Laboratoire Interdisciplinaire des Sciences du Numérique}
\email{viviane.pons@lisn.upsaclay.fr}

\definecolor{darkblue}{rgb}{0,0,0.7} % darkblue color
\definecolor{green}{RGB}{57,181,74} % green color
\newcommand{\darkblue}{\color{darkblue}} % darkblue command
\newcommand{\red}{\color{red}} % red command
\newcommand{\defn}[1]{\emph{\darkblue #1}} % emphasis of a definition

\newtheorem{theorem}{Theorem}%[section]
\newtheorem{corollary}[theorem]{Corollary}
\newtheorem{proposition}[theorem]{Proposition}

\newtheorem{definition}[theorem]{Definition}

\theoremstyle{definition}

\DeclareMathOperator{\dinv}{dinv}
\DeclareMathOperator{\area}{area}
\DeclareMathOperator{\bounce}{bounce}
\DeclareMathOperator{\ins}{ins}
\DeclareMathOperator{\Maxb}{\mathcal{M}}
\DeclareMathOperator{\Maxa}{\mathcal{M}'}
\DeclareMathOperator{\bounces}{\mathcal{B}}

\newcommand{\tdinv}{\emph{dinv}\xspace}
\newcommand{\tarea}{\emph{area}\xspace}
\newcommand{\tbounce}{\emph{bounce}\xspace}

\begin{document}

\maketitle

\begin{abstract}
We give a simple iterative description of the well known Zeta map on Dyck paths which sends the \emph{dinv,area} statistics to the \emph{area,bounce} statistics. Our description uses Dyck paths area sequences and can be implemented easily.
\end{abstract}

\section*{Introduction}

The $qt$-Catalan symmetry is known to be one of the most intriguing questions in algebraic combinatorics. Rooted in the study of diagonal harmonic polynomials~\cite{qtbook}, it leads to numerous combinatorial interpretations and open questions. One of the main problems can be stated in an elementary combinatorial way but has, to this day, no elementary solution. Dyck paths are a well known family of combinatorial objects counted by the Catalan numbers. The \tarea of a Dyck path is a natural statistic leading to common $q$-enumerations of Catalan objects. The \tdinv is another statistic which appears to be more mysterious but has the particularity of having the same distribution as the \tarea. Moreover, there are algebraic proofs that these two statistics are \emph{symmetric}: the number of Dyck paths with \tarea $q$ and \tdinv $t$ is the same as the number of Dyck paths with \tarea $t$ and \tdinv $q$. There is no combinatorial proof of this symmetry, \emph{i.e.}, there is no bijection on Dyck paths which exchange the two statistics. Besides, this symmetry can be extended to more general settings such as triangular Dyck paths~\cite{PRE_BergTriangular} and is conjectured to be schur-positive.

The purpose of this note is to share an algorithm that I have had for some time and which might be useful to the many mathematicians working around this problem. I got an early interest in this question and of course tried to find the mysterious bijection myself. I failed like many others and instead came up with a bijective map that did not switch the statistics but did send the \tarea{} to the \tdinv{}. As I soon discovered, this turned out to the very well known $\zeta$ map described by Haglund~\cite{qtbook} (page 50) and since then generalized and largely studied~\cite{zeta_2014, SweepZeta}. More precisely, this is~$\zeta^{-1}$. As this was no breakthrough and as I could not make much of it, I never properly wrote it until today. 

However, in recent discussions I had, I discovered that many people with strong interest in the original problem were actually not aware of this natural description of the $\zeta$ map. It has the advantage of being an iterative process: it inserts letters inside the \emph{area sequence}. The image of a Dyck path of size $n$ is obtained through a single insertion on the image of the corresponding Dyck path of size $n-1$. As a consequence, it can be implemented easily and we hereby share the {\tt SageMath} implementation directly derived from this paper~\cite{SAGE_areaseq}.

Even though this description seems very natural, I have not seen it anywhere else. There is already an implementation of the $\zeta$ map in {\tt SageMath}~\cite{SageMath2022} using a recursive approach on the area sequence. I tend to believe that the proof that the {\tt SageMath} implementation is indeed the map described by Haglund is rather close to mine but the only reference found in the {\tt SageMath} manual is the Haglund map and no proper proof of the implementation is offered. Besides, the {\tt SageMath} map is slightly different as it works on a reversed Dyck path and is then not directly the $\zeta$ map. This is why I have decided to properly write and prove the present construction of the $\zeta$ map with no other ambition than to share this  piece of knowledge.

The paper is organized in a very straight forward manner. We first present in Section~\ref{sec:area-sequence} the objects at stake, namely: Dyck paths, area sequences, and the three important statistics of the $\zeta$ map \tarea, \tdinv, and \tbounce. The map itself is found in Section~\ref{subsec:area-to-dinv} along with the proof that it is bijective and sends the \tarea of a Dyck path to the \tdinv. In Section~\ref{subsec:bounce-to-area}, we prove that, like the original $\zeta^{-1}$ map, it sends the \tbounce to the \tarea. Finally, in Section~\ref{subsec:zeta}, we show that our map is indeed $\zeta^{-1}$.

In an external annex~\cite{SAGE_areaseq}, we provide the detailed implementation of the map in {\tt SageMap} and a demo {\tt SageMath Jupyter Notebook} following the same structure as the paper, implementing all examples and providing tests for all results. This demo is hosted on a {\tt Github} repo providing a dynamic link allowing to run the code on a distant server without any installation needed.

\section{Area sequences and statistics}
\label{sec:area-sequence}

\subsection{Area sequence of a Dyck path}

A Dyck path is a path in the grid consisting of north steps and east steps starting at $(0,0)$ and such that the path never goes below the line $y = x$. The size of a Dyck path is the number of north steps. Dyck paths are a very classical family of combinatorial objects. The number of Dyck paths of size $n$ is given by the Catalan numbers 

\begin{equation*}
\frac{1}{n+1} \binom{2n}{n}.
\end{equation*}

The \defn{area sequence} of a Dyck path is a word on $\lbrace 0,1,\dots, n-1 \}$ such that the letter in position $i$ gives the number of full $1 \times 1$ cell in the grid to the right of the $i^{th}$ north step and to the left of the $y = x$ line. In other words, these are the cells on the $i^{th}$ line between the path and the $y=x$ line. In Fig.~\ref{fig:dyck_paths3}, we show all Dyck words of size~$3$ with their area sequences. An example of size $17$ is given on Fig.~\ref{fig:dyck_paths_example}.

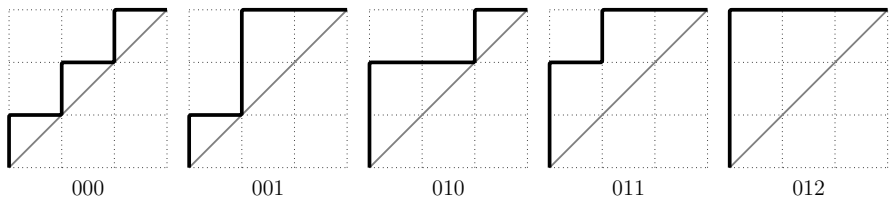
\begin{figure}[ht]
\scalebox{.7}{\begin{tabular}{ccccc}
\begin{tikzpicture}
\draw[dotted] (0, 0) grid (3, 3);
\draw[color = gray, line width = 1] (0,0) -- (3,3);
\draw[rounded corners=1, color=black, line width=2] (0,0) -- (0,1) -- (1,1) -- (1,2) -- (2,2) -- (2,3) -- (3,3);
\end{tikzpicture}
&
\begin{tikzpicture}
\draw[dotted] (0, 0) grid (3, 3);
\draw[color = gray, line width = 1] (0,0) -- (3,3);
\draw[rounded corners=1, color=black, line width=2] (0,0) -- (0,1) -- (1,1) -- (1,2) -- (1,3) -- (2,3) -- (3,3);
\end{tikzpicture}
&
\begin{tikzpicture}
\draw[dotted] (0, 0) grid (3, 3);
\draw[color = gray, line width = 1] (0,0) -- (3,3);
\draw[rounded corners=1, color=black, line width=2] (0,0) -- (0,1) -- (0,2) -- (1,2) -- (2,2) -- (2,3) -- (3,3);
\end{tikzpicture}
&
\begin{tikzpicture}
\draw[dotted] (0, 0) grid (3, 3);
\draw[color = gray, line width = 1] (0,0) -- (3,3);
\draw[rounded corners=1, color=black, line width=2] (0,0) -- (0,1) -- (0,2) -- (1,2) -- (1,3) -- (2,3) -- (3,3);
\end{tikzpicture}
&
\begin{tikzpicture}
\draw[dotted] (0, 0) grid (3, 3);
\draw[color = gray, line width = 1] (0,0) -- (3,3);
\draw[rounded corners=1, color=black, line width=2] (0,0) -- (0,1) -- (0,2) -- (0,3) -- (1,3) -- (2,3) -- (3,3);
\end{tikzpicture} \\
$000$ & $001$ & $010$ & $011$ & $012$
\end{tabular}}
\caption{The Dyck paths of size 3 and their area sequences.}
\label{fig:dyck_paths3}
\end{figure}

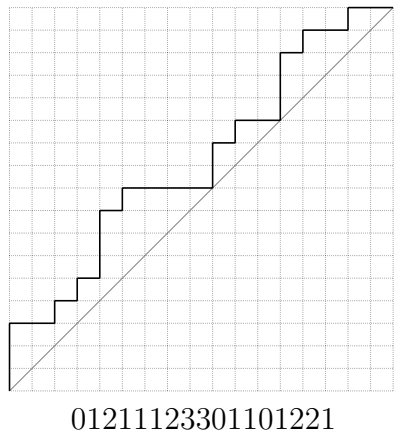
\begin{figure}[ht]
\scalebox{.3}{\begin{tikzpicture}
\draw[dotted] (0, 0) grid (17, 17);
\draw[color = gray, line width = 1] (0,0) -- (17,17);
\draw[rounded corners=1, color=black, line width=2] (0,0) -- (0,1) -- (0,2) -- (0,3) -- (1,3) -- (2,3) -- (2,4) -- (3,4) -- (3,5) -- (4,5) -- (4,6) -- (4,7) -- (4,8) -- (5,8) -- (5,9) -- (6,9) -- (7,9) -- (8,9) -- (9,9) -- (9,10) -- (9,11) -- (10,11) -- (10,12) -- (11,12) -- (12,12) -- (12,13) -- (12,14) -- (12,15) -- (13,15) -- (13,16) -- (14,16) -- (15,16) -- (15,17) -- (16,17) -- (17,17);
\end{tikzpicture}}

$0 1 2 1 1 1 2 3 3 0 1 1 0 1 2 2 1$
\caption{An example of a Dyck path of size 17 with its area sequence.}
\label{fig:dyck_paths_example}
\end{figure}

The following characterization of area sequences is a well known fact. 

\begin{proposition}
A word $w_1 \dots w_n$ of size $n$ on $\lbrace 0,1,\dots, n-1 \}$ is the area sequence of a Dyck path if and only if $w_1 = 0$ and for all $1 \leq i < n$, we have $0 \leq w_{i+1} \leq w_i + 1$.
\end{proposition}

The proof is immediate. The letter $w_i$ corresponds to the number of cells to the right of the $i^{th}$ north step $x_i$. If $x_i$ is not the last north step, it is followed by a certain number $e$ of east steps before the next north step. The conditions on the Dyck path impose that $0 \leq e \leq w_i + 1$ and we see that $w_{i+1} = w_i + 1 - e$. Here are the 14 area sequences for Dyck paths of size 4.

\begin{tabular}{ccccccc}
$0000$ & $0001$ & $0010$ & $0011$ & $0012$ & $0100$ & $0101$ \\
$0110$ & $0111$ & $0112$ & $0120$ & $0121$ & $0122$ & $0123$ 
\end{tabular}

The \defn{area} of a Dyck path is the total number of grid cells between the path and the $y = x$ line. In other words, this is the sum of the values of the area sequence. For example, the \tarea{} of the Dyck Path of Fig.~\ref{fig:dyck_paths_example} is $22$.

\subsection{The \tdinv{}}

The \defn{dinv} statistic on Dyck paths appears in the work of Haglund~\cite{qtbook}. There are many equivalent definitions. We directly give a description on the area sequence.

\begin{definition}
Let $w = w_1 \dots w_n$ be an area sequence. For each letter $w_i$, we call $d_i$ the number of letters $w_j$ with $j > i$ and $w_j = w_i$ or $w_j = w_i - 1$. Then, the \defn{dinv} of w is the sum of the $d_i$.
\end{definition}

For example, for $w = 010$, we obtain $d_1 = 1$, $d_2 = 1$ and $d_3 = 0$, and so $\dinv (w) = 2$. In the example of Fig.~\ref{fig:dyck_paths_example}, the $d_i$ are given by $2, 9, 10, 8, 7, 6, 6, 3, 2, 1, 4, 3, 0, 1, 2, 1$ giving a total \tdinv{} of 65.

\subsection{The \tbounce{}}

Another important statistic when working on the~$\zeta$ map is the \defn{bounce} of a Dyck path. The \tbounce{} can be computed by drawing a ``bounce path'' under the Dyck path. The bounce path is a Dyck path which starts at $(0,0)$ and goes up as much as possible by staying under the original Dyck path, then goes straight to the $y=x$ line and ``bounces back'' again as much as possible as drawn on Fig.~\ref{fig:bounce_path_example}. The area sequence of the bounce path is the \emph{bounce sequence} which can be computed directly from the area sequence of the Dyck path.

\begin{definition}
Let $w = w_1 \dots w_n$ be an area sequence, the bounce sequence $b_1 \dots b_n$ of $w$ is given by
\begin{align*}
b_1 &= 0 \\
b_i &= \begin{cases}
b_{i-1} + 1 & \text{if } b_{i-1} + 1 \leq w_i \\
0 & \text{otherwise.}
\end{cases}
\end{align*}

\end{definition}

Fig.~\ref{fig:bounce_path_example} shows an example of Dyck path with its bounce path and bounce sequence in red. The ``bounces'' of the bounce sequence correspond to all the non initial zeros, we write $\bounces(w) = \lbrace i > 1; b_i = 0 \rbrace$. Now the \defn{bounce} statistic is obtained by the summing the \emph{reversed positions} of the bounces, \emph{i.e.}, their distance to the end of the path which is given by $n - i +1$. For example, in Fig.~\ref{fig:bounce_path_example}, we have $b_4 = b_6 = b_{10} = b_{12} = b_{13} = b_{16} = 0$. We sum all their reversed positions and we obtain $\bounce(w) = 14 + 12 + 8 + 6 + 5 + 2 = 47$.

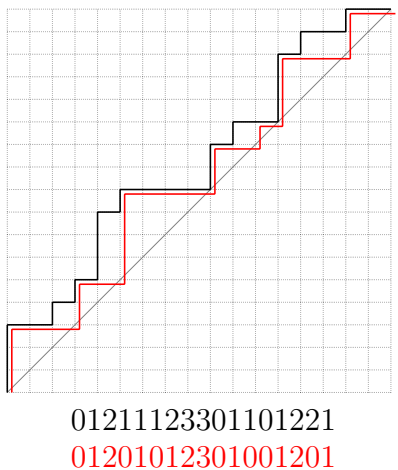
\begin{figure}[ht]
\scalebox{.3}{\begin{tikzpicture}
\draw[dotted] (0, 0) grid (17, 17);
\draw[color = gray, line width = 1] (0,0) -- (17,17);
\draw[rounded corners=1, color=black, line width=2] (0,0) -- (0,1) -- (0,2) -- (0,3) -- (1,3) -- (2,3) -- (2,4) -- (3,4) -- (3,5) -- (4,5) -- (4,6) -- (4,7) -- (4,8) -- (5,8) -- (5,9) -- (6,9) -- (7,9) -- (8,9) -- (9,9) -- (9,10) -- (9,11) -- (10,11) -- (10,12) -- (11,12) -- (12,12) -- (12,13) -- (12,14) -- (12,15) -- (13,15) -- (13,16) -- (14,16) -- (15,16) -- (15,17) -- (16,17) -- (17,17);
\draw[rounded corners=1, color=red, line width=2] (0.2,0) -- (0.2,1) -- (0.2,2) -- (0.2,2.8) -- (1.2,2.8) -- (2.2,2.8) -- (3.2,2.8) -- (3.2,3.8) -- (3.2,4.8) -- (4.2,4.8) -- (5.2,4.8) -- (5.2,5.8) -- (5.2,6.8) -- (5.2,7.8) -- (5.2,8.8) -- (6.2,8.8) -- (7.2,8.8) -- (8.2,8.8) -- (9.2,8.8) -- (9.2,9.8) -- (9.2,10.8) -- (10.2,10.8) -- (11.2,10.8) -- (11.2,11.8) -- (12.2,11.8) -- (12.2,12.8) -- (12.2,13.8) -- (12.2,14.8) -- (13.2,14.8) -- (14.2,14.8) -- (15.2,14.8) -- (15.2,15.8) -- (15.2,16.8) -- (16.2,16.8) -- (17.2,16.8);
\end{tikzpicture}}

     $0 1 2 1 1 1 2 3 3 0 1 1 0 1 2 2 1$ \\
$\red{0 1 2 0 1 0 1 2 3 0 1 0 0 1 2 0 1}$ 
\caption{The bounce path of a Dyck path}
\label{fig:bounce_path_example}
\end{figure}

Note that the external annex~\cite{SAGE_areaseq} provides an implementation of those statistics following the exact definitions of the paper.

\section{The map}

\subsection{Area to dinv}
\label{subsec:area-to-dinv}

In this section, we describe a map which, given an area sequence $w$ with $\sum w = q$, creates a new area sequence $w'$ such that $\dinv(w') = q$. Our algorithm is based on the notion of \defn{insertion}. An \defn{insertion} on an area sequence $w$ at position $i$ is defined as follows

\begin{align}
\ins_0(w) &:= 0 w_1 \dots w_n; \\
\ins_i(w) &:= w_1 \dots w_i (w_i + 1) w_{i+1} \dots w_n 
\end{align}
for $1 \leq i \leq n$. For example, $\ins_0(010) = 0010$, $\ins_1(010) = 0110$, $\ins_2(010) = 0120$, and $\ins_3(010) = 0101$.

It is clear that the result of an insertion on an area sequence is also an area sequence. We are now going to specify a certain list of possible insertions on a given area sequence in order to control the exact \tdinv added by the insertion.

\begin{definition}
\label{def:admissible-insertions}
Let $w$ be an area sequence whose maximum value is~$m$. We write $\Maxb(w) = w^{-1}(m) := \lbrace i; w_i = m \rbrace$, the positions of the letters~$m$ in $w$. We now define $\Maxa(w)$ the set of positions $i$ such that $w_i = m-1$ and for all $j > i$, $w_j < m$. And finally, let us define $i_0(w)$ to be the position of the leftmost letter $m$ in the rightmost block of consecutive $m$ letters. Then the set $\Maxb \cup \Maxa \cup \lbrace i_0 - 1 \rbrace$ are the \defn{admissible insertion positions} of $w$. The \defn{admissible order} on this set is a total order given by the elements of $\Maxb$ in decreasing order followed by the elements of $\Maxa$ in decreasing order followed by $i_0 - 1$. 

If $w$ is the empty word, the admissible insertion positions of $w$ are the set of size $1$ $\lbrace 0 \rbrace$.
\end{definition}

For example, if $w = 0122210122011$, then $\Maxb(w) = \lbrace 3,4,5,9,10 \rbrace$, $\Maxa(w) = \lbrace 12,13 \rbrace$, and $i_0(w) = 9$. The admissible insertion position in admissible order are $10,9,5,4,3,13,12,8$. The annex~\cite{SAGE_areaseq} provides the code to compute more examples.

\begin{proposition}
\label{prop:admissible-insertions}
Let $w$ be an area sequence and $C = c_0, \dots, c_{k} $ its admissible insertion positions taken in admissible order, then 
\begin{equation}
\label{eq:dinv-insertion}
\dinv(\ins_{c_i}(w)) = \dinv(w) + i,
\end{equation}
and the new word $\ins_{c_i}(w)$ has $i+2$ admissible insertion positions. Besides, the inserted value, at position $c_i + 1$, is always of maximal value in $\ins_{c_i}(w)$ and it is the leftmost element of the rightmost block of consecutive maximal values. 
\end{proposition}

\begin{proof}
If $i < |\Maxb(w)|$, then we have $w_{c_i} = m$. For convenience, we write $\tilde{w} := \ins_{c_i}(w)$. When inserting, we obtain $\tilde{w}_{c_i+1} = m+1$. As $m$ was the maximal value of $w$, there are no other letters equal to $m+1$ in $\tilde{w}$ and because we take the positions of $\Maxb(w)$ in decreasing order, there are $i$ letters equal to $m$  to the right of $\tilde{w}_{c_i}$ so $d_{c_i+1} = i$ in $\tilde{w}$ and the rest of the \tdinv{} is left unchanged. Besides, we have $\Maxb(\tilde{w}) = \lbrace c_i + 1 \rbrace$ and $\Maxa(\tilde{w}) = \lbrace c_{i-1} + 1, c_{i-2} +1, \dots, c_0 +1\rbrace$, \emph{i.e.}, the shifted positions of the letters $m$ in $w$ which were to the right of the position $c_i$, and $i_0(\tilde{w}) = c_i -1$. The number of admissible insertion positions in $\tilde{w}$ is then given by $i + 2$. Finally, the inserted letter is of maximal value and as it is the only letter with such value, it has the desired properties.

If $| \Maxb(w) | \leq i < |\Maxb(w)| + |\Maxa(w)|$, then we have $w_{c_i} = m -1$. When inserting, we obtain $\tilde{w}_{c_i + 1} = m$. Note that in this case, $i = | \Maxb(w) | + j$ with $j \geq 0$. Besides, by definition of the admissible order, $c_i$ is greater than all the elements of $\Maxb(w)$. This means that for all $k \in \Maxb$, $d_k$ is increased by one in $\tilde{w}$ and there are no elements equal to $m$ to the right of $\tilde{w}_{c_i+1}$. On top of that, as we take the values of $\Maxa(w)$ in decreasing order, there are $j$ elements in $\tilde{w}$ with value $m-1$ to the right of $\tilde{w}_{c_i +1}$. So $d_{c_i+1} = j$ and in total, the \tdinv{} is increased by $| \Maxb(w) | + j = i$. Besides, $|\Maxb(\tilde{w})| = |\Maxb(w)| + 1$ whereas $|\Maxa(\tilde{w})| = j$ and so $\tilde{w}$ has $|\Maxb(w)| + 1 + j + 1= i +2$ admissible insertion positions. Finally, the inserted letter is of value $m$, which is still the maximal value in $\tilde{w}$, it is to the right of all the other letters $m$ and forms by itself a block of a single element.

Now, the last case is when the insertion occurs at $i_0 - 1$. Remember that $i_0$ is the leftmost letter $m$ in the rightmost block of consecutive letters $m$ in $w$. It is possible that $i_0 = 1$ is the first letter of $w$. In this case, $w_{i_0}$ as well as  all letters of $w$ are actually equal to $0$, and after the insertion, we obtain $\tilde{w}_{i_0} = \tilde{w}_0 = 0 = m$. Otherwise, as $w_{i_0} = m$, we have $w_{i_0 - 1} \geq m - 1$. By definition of $i_0$, $w_{i_0 - 1}$ cannot be equal to $m$ and because $m$ is maximal among the letters of $w$, we obtain $w_{i_0 - 1} = m - 1$. After the insertion, we have $\tilde{w}_{i_0} = m$. For all $j \in \Maxb(w)$ such that $j < i_0$, the value $d_j$ increases by $1$ whereas $d_{i_0} = |\lbrace j \in \Maxb(w); j > i_0 \rbrace| + \Maxa(w)$ so the \tdinv{} is indeed increased by $|\Maxb(w)| + |\Maxa(w)|$.  Besides, $|\Maxb(\tilde{w})| = |\Maxb(w)| + 1$ whereas $\Maxa(\tilde{w}) = \Maxa(w)$ and so $\tilde{w}$ has $|\Maxb(w)| + |\Maxa(w)| + 2$ admissible insertion positions. Finally, the inserted letter is of maximal value and by definition of $i_0$, it is the leftmost of the rightmost block of consecutive maximal letters. 
\end{proof}

To illustrate this proof, we show all admissible insertions on the area sequence $0122210122011$ on Fig.~\ref{fig:admissible-insertions}. The \tdinv{} of the sequence is $44$. We have marked the admissible insertion positions with $\cdot_i$ where $i$ is the index in the admissible order. You can check that an insertion at $c_i$ corresponds to an increase of $i$ in the \tdinv{} and that the new sequence has $i+2$ insertion positions. 

\begin{figure}[ht]
\begin{tabular}{llc}
insertion position & area sequence & \tdinv{} \\
  & $0~1~2 \cdot_4 2 \cdot_3 2 \cdot_2 1~0~1 \cdot_7 2 \cdot_1 2 \cdot_0 0~1 \cdot_6 1 \cdot_5 $ & 44 \\
$c_0 = 10$ & $0~1~2~2~2~1~0~1~2~2 \cdot_1 {\red 3} \cdot_0 0~1~1$ & 44 \\
$c_1 = 9$  & $0~1~2~2~2~1~0~1~2 \cdot_2 {\red 3} \cdot_0 2 \cdot_1 0~1~1$ & 45 \\
$c_2 = 5$  & $0~1~2~2~2 \cdot_3 {\red 3} \cdot_0 1~0~1~2 \cdot_2 2 \cdot_1 $ & 46 \\
$c_3 = 4$  & $0~1~2~2 \cdot_4 {\red 3} \cdot_0 2 \cdot_3 1~0~1~2 \cdot_2 2 \cdot_1 0~1~1$ & 47 \\
$c_4 = 3$  & $0~1~2 \cdot_5 {\red 3} \cdot_0 2 \cdot_4 2 \cdot_3 1~0~1~2 \cdot_2 2 \cdot_1 0~1~1$ & 48 \\
$c_5 = 13$ & $0~1~2 \cdot_5 2 \cdot_4 2 \cdot_3 1~0~1~2 \cdot_2 2 \cdot_1 0~1~1 \cdot_6 {\red 2} \cdot_0 $ & 49 \\
$c_6 = 12$ & $0~1~2 \cdot_5 2 \cdot_4 2 \cdot_3 1~0~1~2 \cdot_2 2 \cdot_1 0~1 \cdot_7 {\red 2} \cdot_0 1 \cdot_6$ & 50 \\
$c_7 = 8$  & $0~1~2 \cdot_5 2 \cdot_4 2 \cdot_3 1~0~1 \cdot_8 {\red 2} \cdot_2 2 \cdot_1 2 \cdot_0 0~1 \cdot_7 1 \cdot_6$ & 51
\end{tabular}
\caption{All admissible insertions on a given area sequence.}
\label{fig:admissible-insertions}
\end{figure}

\begin{theorem}
\label{thm:area_dinv}
Let $w$ be an area sequence, we define $\psi(w)$ recursively as follows. If $w = \varepsilon$, then $\psi(w) := \varepsilon$. Otherwise, $w = ua$ where $u$ is a word and $a$ is a letter, and $\psi(w) := \ins_{c_a}(\psi(u))$ where $c_0, c_1, \dots, c_k$ are the admissible insertion positions of $\psi(u)$ taken in admissible order. 

Then $\psi$ defines a bijection on area sequences such that $\dinv(\psi(w)) = \area(w)$. 
\end{theorem}

In other words, you read the area sequence from left to right and each letter is giving you an insertion to perform on the image using the admissible insertion positions.

\begin{proof}
This is immediate by induction using Proposition~\ref{prop:admissible-insertions}. Indeed, you first need to check that the map is well defined, \emph{i.e.} that $a \leq k$, $k+1$ being the number of admissible insertion positions of $\psi(u)$. It is true on the initial case: the first of letter of $w$ is $0$ and there is $1$ admissible insertion position on the empty word. Then Proposition~\ref{prop:admissible-insertions} ensures that after inserting at $c_a$, you obtain $a+2$ admissible insertion positions corresponding to the $a+2$ possible letters that could follow $a$ (from $0$ to $a+1$). Then \eqref{eq:dinv-insertion} gives you the expected result for the \tdinv{}. Besides, it is indeed a bijection as the operation is reversible: by Proposition~\ref{prop:admissible-insertions}, the last inserted letter is the leftmost letter in the rightmost block of maximal values. By removing this letter from a word $w'$, the difference in \tdinv{} tells you by which letter $a$ the preimage $w$ of $w'$ ends.  
\end{proof}

Fig.~\ref{fig:example_bijection} shows the step by step computation of the word of Fig.~\ref{fig:dyck_paths_example}. At each step, the admissible insertion positions are indicated with their order on the image $\psi$.

\begin{figure}[ht]
\begin{tabular}{ll}
$0$        & $\cdot_1 0 \cdot_0$ \\
{\red $0$} & \\ \hline
$01$        & $\cdot_2 0 \cdot_1 0 \cdot_0$ \\
{\red $01$} & \\ \hline
$012$        & $ \cdot_3 0 \cdot_2 0 \cdot_1 0 \cdot_0$ \\
{\red $012$} & \\ \hline
$0121$        & $0 ~ 0 \cdot_2 1 \cdot_0 0 \cdot_1$ \\
{\red $0120$} & \\ \hline
$01211$        & $0 ~ 0 ~ 1 \cdot_1 0 \cdot_2 1 \cdot_0$ \\
{\red $01201$} & \\ \hline
      $012111$   & $0 ~ 0 ~ 1 \cdot_2 2 \cdot_0 0 ~ 1 \cdot_1$ \\
{\red $012010$} & \\ \hline
      $0121112$ & $0 ~ 0 ~ 1 \cdot_3 2 \cdot_1 2 \cdot_0 0 ~ 1 \cdot_2$ \\
{\red $0120101$}  & \\ \hline
      $01211123$ & $0 ~ 0 ~ 1 \cdot_4 2 \cdot_2 2 \cdot_1 2 \cdot_0 0 ~ 1 \cdot_3$ \\
{\red $01201012$}  & \\ \hline
      $012111233$ & $0 ~ 0 ~ 1 ~ 2 \cdot_3 2 \cdot_2 2 \cdot_1 0 ~ 1 \cdot_4 2 \cdot_0$ \\
{\red $012010123$}  & \\ \hline
      $0121112330$ & $0 ~ 0 ~ 1 ~ 2 ~ 2 ~ 2 ~ 0 ~ 1 ~ 2 \cdot_1 3 \cdot_0$ \\
{\red $0120101230$}  & \\ \hline
      $01211123301$ & $0 ~ 0 ~ 1 ~ 2 ~ 2 ~ 2 ~ 0 ~ 1 ~ 2 \cdot_2 3 \cdot_1 3 \cdot_0$ \\
{\red $01201012301$}  & \\ \hline
      $012111233011$ & $0 ~ 0 ~ 1 ~ 2 ~ 2 ~ 2 ~ 0 ~ 1 ~ 2 ~ 3 \cdot_2 4 \cdot_0 3 \cdot_1$ \\
{\red $012010123010$}  & \\ \hline
      $0121112330110$ & $0 ~ 0 ~ 1 ~ 2 ~ 2 ~ 2 ~ 0 ~ 1 ~ 2 ~ 3 ~ 4 \cdot_1 5 \cdot_0 3$ \\
{\red $0120101230100$}  & \\ \hline
      $01211123301101$ & $0 ~ 0 ~ 1 ~ 2 ~ 2 ~ 2 ~ 0 ~ 1 ~ 2 ~ 3 ~ 4 \cdot_2 5 \cdot_1 5 \cdot_0 3$ \\
{\red $01201012301001$}  & \\ \hline
      $012111233011012$ & $0 ~ 0 ~ 1 ~ 2 ~ 2 ~ 2 ~ 0 ~ 1 ~ 2 ~ 3 ~ 4 \cdot_3 5 \cdot_2 5 \cdot_1 5 \cdot_0 3$ \\
{\red $012010123010012$}  & \\ \hline
      $0121112330110122$ & $0 ~ 0 ~ 1 ~ 2 ~ 2 ~ 2 ~ 0 ~ 1 ~ 2 ~ 3 ~ 4 ~ 5 \cdot_3 6 \cdot_0 5 \cdot_2 5 \cdot_1 3$ \\
{\red $0120101230100120$}  & \\ \hline
      $01211123301101221$ & $0 ~ 0 ~ 1 ~ 2 ~ 2 ~ 2 ~ 0 ~ 1 ~ 2 ~ 3 ~ 4 ~ 5 ~ 6 \cdot_1 5 ~ 5 \cdot_2 6 \cdot_0 3$ \\
{\red $01201012301001201$}  & \\ \hline
\end{tabular}
\caption{The step by step $\psi$ map on an example. On the left, an area sequence $w$ with its bouncing path underneath. On the right, $\psi(w)$ with its admissible insertion positions.}
\label{fig:example_bijection}
\end{figure}

\subsection{Bounce to area}
\label{subsec:bounce-to-area}

In this section, we explain how the \tbounce{} statistic can be understood through the $\psi$ map. This will also be key to prove that this map is actually $\zeta^{-1}$.

\begin{proposition}
\label{prop:bounce}
Let $w$ be an area sequence of size $n$ and $\psi$ the map defined in Theorem~\ref{thm:area_dinv}. We write $b_1, \dots b_n$ the bounce sequence of $w$ and $\bounces(w)$ the bounces of $w$, \emph{i.e.} the non initial $0$ of the bounce sequence. Remember that the \tbounce{} of $w$ is given by the sum of the reversed position $n - i + 1$ for all $i$ in $\bounces(w)$. We have
\begin{itemize}
\item $b_n = 0$ if and only if $\max(\psi(w)) = \max(\psi(u)) + 1$, where $u$ is the prefix of $w$ of size $n-1$;
\item $|\bounces(w)| = \max(\psi(w))$
\item $|\Maxb(\psi(w))| = b_n +1$;
\item $|\Maxa(\psi(w))| = w_n - b_n$;
\end{itemize}
where $\Maxa$ and $\Maxb$ are taken from Definition~\ref{def:admissible-insertions}. Besides, for each $1 \leq k \leq \max(\psi(w))$, the number of values greater than or equal to $k$ in $\psi(w)$ is given by the reversed position of the $k^{th}$ bounce of the the bounce sequence of $w$.
\end{proposition}

For example, look at the image of $w = 01211123301101221$ which can be found on the last line of Fig.~\ref{fig:example_bijection}. As before, we write $b_1 \dots b_n$ its bounce sequence given in red underneath $w$. We have that $b_n > 0$ and the last inserted value is a $6$ which is equal to $\max(\psi(u))$. We have $|\bounces(w)| = 6$ because the bounce sequence has $6$ non initial zeros. $|\Maxb(\psi(w))|$ corresponds to the number of letters equal to $6$ in $\psi(w)$, there are indeed $2 = b_n +1$ of them. And $|\Maxa(\psi(w)|$ is the number of letters equal to $5$ to the right of the last $6$, there are none of them and indeed $w_n - b_n = 0$. Besides, the first bounce of $w$ is $b_4 = 0$, with reversed position $14$, and there are $14$ values greater than or equal to $1$ in $\psi(w)$. The second bounce is $b_6 = 0$ with reversed position $12$ which is the number of values greater than or equal to $2$ in $\psi(w)$. This can be checked for all the other bounces of $w$. 

\begin{proof}
Again, the proof can be made by induction. First note that the area sequence $0$ of size $1$ satisfies all the conditions. In this case, the bounce sequence is also $0$ as well as $\psi(w)$. We have $\bounces(w) = \emptyset$, $\Maxb(\psi(w)) = \lbrace 1 \rbrace$ and $\Maxa(\psi(w)) = \emptyset$.

Now take $w$ an area sequence of size $n > 0$ and suppose that it satisfies the conditions. We prove that the area sequence $wa$ with $0 \leq a \leq w_n+1$ also satisfies it. We use all the previous notations and write $m := \max(\psi(w))$ for convenience.

Let us look at the case where $a < |\Maxb(\psi(w))|$. In particular, this means that $b_n + 1 > a$. By definition of the bounce sequence, this implies $b_{n+1} = 0$. Remember from the proof of Proposition~\ref{prop:admissible-insertions} that if we take $a < |\Maxb(\psi(w))|$, the inserted letter in $\psi(w)$ is going to be $m + 1$. So we have indeed that $b_{n+1} = 0$ and $\max(\psi(wa)) = m + 1$. Besides $\bounces(wa) = \bounces(w) \cup \lbrace n+1 \rbrace$ and so $|\bounces(wa)| = \max(\psi(wa))$. Now as we added a letter $m + 1$, for all $k \leq m$ we have increased by $1$ the number of letters in the image with a value greater than or equal to $k$. As the length of the bounce sequence has increased, the reversed positions of all the elements of $\bounces(w)$ has also been increased by one. And there is a unique value $m + 1$ which corresponds to the reversed position $1$ of $b_{n+1}$. This ensures that the proposition is satisfied. We still need to check that the conditions on $\Maxb$ and $\Maxa$ are still satisfied on $wa$. We have indeed that $|\Maxb(wa)| = 1 = b_{n+1} + 1$. By definition of the map and of the admissible insertion order, the value of $a$ indicates the number of letters equal to $m$ which will end up to the right of the inserted $m+1$ in $\psi(wa)$. So we have $|\Maxa(\psi(wa)| = a = a - b_{n+1}$, because $b_{n+1} = 0$. 

In the case where $a > |\Maxb(\psi(w))|$, the letter inserted in $\psi(w)$ is $m$. On the bounce sequence, we have $a > b_n + 1$ and so $b_{n+1} = b_n + 1 > 0$, and $|\bounces(wa)| = |\bounces(w)| = m = \max(\psi(wa))$. As before, for all $k \leq m$ we have increased by $1$ the number of letters in the image with a value greater than or equal to $k$. As the length of the bounce sequence has increased, the reversed positions of all the elements of $\bounces(wa) = \bounces(w)$ has also been increased by one and the proposition is satisfied. Now $|\Maxb(\psi(wa)| = |\Maxb(\psi(w))| + 1 = b_n + 1 + 1 = b_{n+1} + 1$ as we have added one letter of maximal value $m$. Finally, by definition of the map, we have $a = |\Maxb(w)| + j$ where $j$ is the number of letters equal to $m-1$ to the right of the newly inserted value $m$ in $\psi(wa)$. We obtain $|\Maxa(\psi(wa))| = j = a - |\Maxb(\psi(w))| = a - (b_n + 1) = a - b_{n+1}$.
\end{proof}

\begin{corollary}
Let $w$ be an area sequence and $\psi$ the map defined in Theorem~\ref{thm:area_dinv}. Then
\begin{equation}
\bounce(w) = \area(\psi(w)).
\end{equation}
\end{corollary}

\begin{proof}
This is immediate: the proposition ensures that the sum of the area sequence $\psi(w)$ is the sum of the reversed positions of the bounce sequence of $w$.
\end{proof}

\subsection{The $\zeta$ map}
\label{subsec:zeta}

Haglund describes the $\zeta$ map in~\cite{qtbook}, page 50. In particular, it has the property that for a Dyck path $w$, $\dinv(w) = \area(\zeta(w))$ and $\area(w) = \bounce(\zeta(w))$. We have seen that this is also the case for $\psi^{-1}$ and we can prove that these are the same maps.

\begin{proposition}
The $\psi$ map is the inverse of the classical $\zeta$ map defined by Haglund.
\end{proposition}

\begin{proof}
This is a direct consequence of Proposition~\ref{prop:bounce}. Indeed, in his book Haglund describes the image of a path $\pi$ though $\zeta$  by first constructing the bounce path of $\pi$. He gives the bounce steps $\alpha_1, \dots, \alpha_k$. The bounces of the pre-image of $\pi$ through $\psi$ as described in Proposition~\ref{prop:bounce} are just a reformulation of those and follow the same rule. 

Now the path $\zeta(\pi)$ between two peaks of the bounce path is entirely determined by the relative positions of the occurrences of two consecutive values in the area sequence of $\pi$ (first the $0$ and $1$, then the $1$ and $2$, etc.). Similarly, each new $0$ in the bounce path of an area sequence $w$ creates a new maximum $m$ in $\psi(w)$, the values of $w$ before the next bounce will determine the relative placements of the $m$ and $m-1$ values in a way that is similar to Haglund's construction. 
\end{proof}

\section*{Acknowledgments} 

This work has received funding from the PAGCAP Project ANR-21-CE48-0020. I also thank the CNRS and INSMI for the opportunity to spend a year at IRL CRM in Montreal. Thank you to Hugh Thomas and Nathan Williams for looking at the paper before its release.

\bibliographystyle{alpha}
\bibliography{../../biblio/all}

\end{document}